\documentclass[12pt,a4paper]{amsart} 
\pdfoutput=1
\usepackage[utf8]{inputenc} 
\usepackage[top=3cm, bottom=3.5cm, left=2.5cm, right=2.5cm]{geometry}
\usepackage{amsmath,amsthm,amsfonts,amstext,amssymb}
\usepackage{enumerate}
\usepackage{hyperref}
\usepackage{color}
\usepackage{xcolor}
\hypersetup{
    colorlinks,
    linkcolor={blue!50!black},
    citecolor={blue!50!black},
    urlcolor={blue!80!black}
}

\def\R{{\mathbb{R}}}
\def\N{{\mathbb{N}}}
\def\Z{{\mathbb{Z}}}
\def\P{{\mathbb P}} 

\def \torus{{{\mathbb T}^d_N}} 
\def \range{{{\mathcal I}^u_N}} 
\def \I{{\mathcal I}} 
\def \mix{t_{\mathrm{mix}}}

\def \GG{{\mathbb G}} 
\def \lscale{{\lambda}} 

\def \ballZ{{\mathrm B}} 
\def \atom{\omega} 
\def \set{{\mathcal S}} 

\def \cemax{{\widetilde {\mathcal C}}}

\newtheorem{theorem}{Theorem}[section]
\newtheorem{corollary}[theorem]{Corollary}
\newtheorem{lemma}[theorem]{Lemma}

\theoremstyle{definition}

\newtheorem{definition}[theorem]{Definition}
\newtheorem{remark}[theorem]{Remark}

\numberwithin{equation}{section}

\begin{document}

\title{Mixing time for the random walk on the range of the random walk on tori}

\author{Jiří Černý}
\address{
  Jiří Černý,
  University of Vienna, Faculty of Mathematics,
  Oskar-Morgenstern-Platz 1, 1090 Vienna, Austria.
}
\email{jiri.cerny@univie.ac.at}

\author{Artem Sapozhnikov}
\address{
  Artem Sapozhnikov,
  University of Leipzig, Department of Mathematics,
  Augustusplatz 10, 04109 Leipzig, Germany.
}
\email{artem.sapozhnikov@math.uni-leipzig.de}

\begin{abstract}
  Consider the subgraph of the discrete $d$-dimensional torus of size
  length $N$, $d\geq 3$, induced by the range of the simple random walk
  on the torus run until the time $uN^d$. We prove that for all $d\geq 3$
  and $u>0$, the mixing time for the random walk on this subgraph is
  of order $N^2$ with probability at least $1 - Ce^{-(\log N)^2}$.
\end{abstract}


\maketitle

\section{Introduction}

Let $X_n$ be a simple random walk  on a large $d$-dimensional discrete
torus $\torus = (\Z/N\Z)^d$, $d\geq 3$, started from the uniform
distribution on $\torus$. For $u>0$ and $N\in\N$, let
\[
  \range = \left\{X_0,\ldots,X_{\lfloor uN^d\rfloor}\right\}
\]
be its range on the time interval $[0,uN^d]$. We view $\range$ as a
subgraph of $\torus$ in which the edges are drawn between any two
vertices  within $\ell^1$-distance $1$ from each other.

In this note, we are interested in the behavior of the mixing time of the
random walk on this graph as $N$ grows while $u>0$ remains fixed. We
prove that the mixing time is of order $N^2$ and give bounds on the
probability of the good event.

To state our main theorem, we recall that a lazy random walk on a finite
connected graph $G=(V,E)$ is a Markov chain with the transition
probabilities $\{p(x,y)\}_{x,y\in V}$ given by
\[
  p(x,y) = \begin{cases}
    \frac 12 & \text{if }x=y,\\
    \frac{1}{2 d_x} & \text{if }|x-y|_1=1, \\
    0 & \text{otherwise},
  \end{cases}
\]
where $d_x$ is the degree of $x$ in $G$. The
\emph{$\frac14$-uniform mixing time} (or simply mixing time)
of the lazy random walk on $G$ is defined by
\[
  \mix(G) = \min\left\{n\,:\,\left|\frac{p_n(x,y) -
      \pi(y)}{\pi(y)}\right|\leq \frac14,~\text{for all }x,y\in V\right\},
\]
where $\pi$ denotes the (unique) stationary
distribution of the walk, and $p_n$ its $n$-step transition probability.

Our main result is the following theorem.

\begin{theorem}\label{thm:mixingtime}
  Let $d\geq 3$ and $u>0$. There exist $c=c(d,u)>0$ and
  $C=C(d,u)<\infty$ such that for all $N\in\N$,
  \begin{equation}\label{eq:mixingtime}
    \P\left[cN^2\leq \mix(\range)\leq CN^2\right]\geq 1 - Ce^{-(\log N)^2}.
  \end{equation}
\end{theorem}

The lower bound on $\mix(\range)$ of Theorem~\ref{thm:mixingtime} is not
difficult to show. In fact, its probability can be easily improved to
$\geq 1 - Ce^{-N^\delta}$. The substantial contribution of this note is
the upper bound on $\mix(\range)$ of correct order. Previously, it was
shown by Procaccia and Shellef in \cite[Theorem~2.2]{ProcacciaShellef} that
\[
  \lim_{N\to\infty}\P\left[\mix(\range)\leq N^2\log^{(k)}N\right]=1,
  \qquad \text{for every $k\ge 0$,}
\]
where $\log^{(k)}N$ is the $k$-th iterated logarithm. Our theorem on the
one hand sharpens their result, and on the other hand gives a bound on
the probability of the good event. The decay rate in \eqref{eq:mixingtime}
can be easily improved from $e^{-(\log N)^2}$ to any $e^{-(\log N)^p}$,
$p>2$, but our method does not allow to obtain a stretched exponential
rate $e^{-N^\delta}$.

\bigskip

The main ingredient of the proof of Theorem~\ref{thm:mixingtime} is the
following isoperimetric inequality for subsets of $\range$, which may be
of independent interest. For $A\subseteq \range$, let
\[
  \partial_\range A
  = \left\{\{x,y\}~:~ x\in A, y\in \range\setminus A, |x-y|_1=1\right\}
\]
be the edge boundary of $A$ in $\range$.

\begin{theorem}\label{thm:isopineq:range}
  Let $d\geq 3$, $u>0$, and $\mu\in(0,1)$.
  There exist $\gamma = \gamma(d,u,\mu)>0$ and $C = C(d,u,\mu)<\infty$ such that for all $N\in\N$,
  \begin{equation}\label{eq:isopineq:range}
    \mathbb P\left[\begin{array}{c}
        \text{for any $A\subset \range$ with $|A|\leq \mu|\range|$,}\\
        \text{$|\partial_\range A|\geq \gamma\cdot |A|^{1 - \frac{1}{d} + \frac{1}{d^2}}\cdot N^{-\frac{1}{d}}$}
    \end{array}\right]\geq 1 - Ce^{-(\log N)^2}.
  \end{equation}
\end{theorem}

Theorem~\ref{thm:isopineq:range} is proved by combining a new
isoperimetric inequality for (deterministic) graphs from
\cite{Sapozhnikov14} with the strong coupling of
the $\range$ and the random interlacements from \cite{CernyTeixeira14}.
We will recall these results in Section~\ref{sec:inputs}.
In the remaining two sections of this note we then prove
Theorem~\ref{thm:isopineq:range} and Theorem~\ref{thm:mixingtime},
respectively.

In the remainder of this note, we omit the dependence of various
constants on $d$. The constants inherit their numbers from the theorems
where they appear for the first time, and their dependence on other
parameters is explicitly mentioned.

\section{Preliminaries} 
\label{sec:inputs}

We introduce some notation first. For $x=(x_1,\dots,x_d)\in \R^d$, its
$\ell^1$ and $\ell^\infty$ norms are defined by
$|x|_1 = \sum_{i=1}^d|x_i|$ and $|x|_\infty = \max\{|x_1|,\ldots|x_d|\}$,
respectively. For $x \in \Z^d$ and $r>0$, we denote by
$\ballZ(x,r) = \{y\in\Z^d~:~|x-y|_\infty\leq \lfloor r \rfloor \}$ the
closed $\ell^{\infty}$-ball in $\Z^d$ with radius $\lfloor r \rfloor$ and
center at $x$. For two subsets of $\Z^d$, $A\subseteq B$, we denote the
boundary of $A$ in $B$ by
\[
\partial_B A = \left\{\{x,y\}~:~x\in A,~y\in B\setminus A,~|x-y|_1 = 1\right\}.
\]

We consider the measurable space $\Omega = \{0,1\}^{\Z^d}$, $d\geq 3$,
equipped with the $\sigma$-algebra $\mathcal F$ generated by the coordinate
maps $\{\atom\mapsto\atom(x)\}_{x\in\Z^d}$. For any
$\atom\in\{0,1\}^{\Z^d}$, we denote the induced subset of $\Z^d$ by
\[
  \set = \set(\atom) = \{x\in\Z^d~:~\atom(x) = 1\} \subseteq \Z^d .\
\]
We view $\set$ as a subgraph of $\Z^d$ in which the edges are drawn
between any two vertices of $\set$ within $\ell^1$-distance $1$ from each
other.

\subsection{Deterministic isoperimetric inequality}

One of the main tools for our proofs is an isoperimetric inequality from \cite{Sapozhnikov14}
for subsets of a (deterministic) graph, satisfied uniformly over a large class of graphs.
Each graph in this class is contained in a large box, well-connected on a mesoscopic scale,
and admits a dense well-structured connected subset identified through a multiscale renormalization scheme.
In this section we recall some notation and necessary results from \cite{Sapozhnikov14}.

Let $\lscale$ and $L_0$ be positive integers. For $n\ge 0$
we consider the following sequences of scales
\begin{equation} \label{def:scales}
  l_n = \lambda^2\, 4^{n^2}, \qquad r_n = \lambda\, 2^{n^2},\qquad L_{n+1}
  = l_n\, L_n.
\end{equation}
To each $L_n$ we associate the rescaled lattice
\[
  \GG_n = L_n\cdot \Z^d = \left\{L_n\cdot x ~:~ x\in\Z^d\right\} ,\
\]
with edges between any pair of $\ell^1$-nearest neighbor vertices of
$\GG_n$. Let
$\eta = (\eta_1,\eta_2)$ be an ordered pair of real numbers satisfying
\begin{equation}\label{eq:etas}
\eta_1\in(0,1),\quad \eta_1 \leq \eta_2 <2\, \eta_1.
\end{equation}

To set up a multiscale renormalization with scales $L_n$, we introduce two families of good vertices.
Their precise definition will not be used in the paper. The reader may skip directly to the statement of Theorem~\ref{thm:isop:QKs}.

\begin{definition}
  We say that $x\in\GG_0$ is \emph{$(0a)$-good} in configuration
  $\omega\in\Omega$ if
    for each $y\in\GG_0$ with $|y-x|_1 \leq L_0$, the set
    $\set\cap(y+[0,L_0)^d)$ contains a connected component $\mathcal C_y$
    with at least $\eta_1 L_0^d$ vertices such that
    for all $y\in\GG_0$ with $|y-x|_1 \leq L_0$, $\mathcal C_y$ and
    $\mathcal C_x$ are connected in
    $\set \cap((x+[0,L_0)^d)\cup(y+[0,L_0)^d))$.

  If $x\in\GG_0$ is not \emph{$(0a)$-good}, then we call it \emph{$(0a)$-bad}.
  For $n\geq 1$, we recursively define $x\in\GG_n$ to be $(na)$-\emph{bad}
  in $\omega\in\Omega$ if there exist two $((n-1)a)$-bad vertices
  $x_1,x_2\in \GG_{n-1}\cap(x+[0,L_n)^d)$ with
  $|x_1 - x_2|_\infty \geq r_{n-1} L_{n-1}$. Otherwise, we call the
  vertex $x$ $(na)$-\emph{good}.
\end{definition}

\begin{definition}
  We say that $x\in\GG_0$ is $(0b)$-\emph{good} in configuration
  $\omega\in\Omega$ if
  \begin{equation*}
      \left|\set\cap(x+[0,L_0)^d)\right| \leq \eta_2 L_0^d.
    \end{equation*}

  If $x\in\GG_0$ is not $(0b)$-\emph{good}, then we call it $(0b)$-\emph{bad}.
  For $n\geq 1$, we recursively define $x\in\GG_n$ to be $(nb)$-\emph{bad}
  in $\omega\in\Omega$ if there exist two $((n-1)b)$-bad vertices
  $x_1,x_2\in \GG_{n-1}\cap(x+[0,L_n)^d)$ with
  $|x_1 - x_2|_\infty \geq r_{n-1}\, L_{n-1}$. Otherwise, we call the
  vertex $x$ $(nb)$-\emph{good}.
\end{definition}

\begin{definition}
  For $n\geq 0$, we say that $x\in\GG_n$ is \emph{$n$-good} in
  configuration $\omega\in\Omega$ if it is at the same time $(na)$ and
  $(nb)$ good. Otherwise, we call the vertex $x$ $n$-\emph{bad}.
\end{definition}

Let us briefly comment on the above definitions.
In classical renormalization techniques on percolation clusters, good boxes are
usually defined as the ones containing a unique cluster with large diameter.
In our case, it is crucial to define good boxes in terms of only monotone events,
as for these one can get a good control of correlations with a help of sprinkling, even in
models with polynomial decay of correlations (see \cite{SznitmanAM, Sznitman:Decoupling, PopovTeixeira}).

The main motivation behind the choice of $\eta_2<2\eta_1$ comes from the following observation.
If neighbors $x,y\in\GG_0$ are $0$-good, then $\mathcal C_x$ and $\mathcal C_y$ are defined uniquely and locally connected
(see \cite[Lemma~3.1]{Sapozhnikov14}).
This property is essential to identify a ubiquitous well-structured connected subset of $\set$ through a multiscale renormalization.
(See \cite{RS:Disordered, DRS12, PRS, Sapozhnikov14}.)

\def\cemax{{\widetilde{\mathcal C}}}

The following statement is a special case of \cite[Lemma~3.3 and Theorem~3.8]{Sapozhnikov14}.
It gives an isoperimetric inequality for subsets of a local enlargement of $\set\cap[0,KL_s)^d$.
This enlargement serves as a smoothening of a possibly rough boundary of $\set\cap[0,KL_s)^d$,
thus improving isoperimetric properties of $\set\cap[0,KL_s)^d$ near its boundary.
\begin{theorem}\label{thm:isop:QKs}
  Let $d\geq 3$ and $\eta $ as above. There exist $C = C(\eta )<\infty$,
  $\beta = \beta(\eta )>0$,
  and $\gamma=\gamma(\eta )>0$ such that for all $\lambda\geq C$, $L_0\geq 1$,
  $s\geq 0$, and  $K\geq 2 L_s^d$, the following statement holds: If $\omega\in\Omega$
  satisfies
  \begin{enumerate}[(a)]
    \item
    all the vertices in $\GG_s\cap [-2L_s,(K+2)L_s)^d$ are $s$-good,
    \item
    all $x,y\in\set\cap[0,KL_s)^d$ with $|x-y|_\infty\leq L_s$ are
    connected in $\set\cap\ballZ(x,2L_s)$,
    \item
    for every $x\in[0,KL_s)^d$ such that $(x+[0,L_s)^d)\subset [0,KL_s)^d$,
    $\set\cap(x+[0,L_s)^d)\neq \emptyset$,
  \end{enumerate}
  then
  \[
    |\set\cap [0,KL_s)^d|\geq \beta\cdot  (KL_s)^d,
  \]
  the set
  \[
  \cemax = \left\{x\in\set~:~\text{$x$ is connected to some $y\in\set\cap[0,KL_s)^d$ by a path in $\set\cap\ballZ(y,2L_s)$}\right\}
  \]
is connected, and for all $A\subset \cemax$ with $L_s^{d(d+1)}\leq |A|\leq \frac12\, |\cemax|$,
\[
|\partial_\cemax A| \geq \gamma\cdot |A|^{\frac{d-1}{d}}.
\]
\end{theorem}

\begin{remark}
\begin{enumerate}[(a)]
\item
There is a small difference between our definition of $0$-good vertices
and that of \cite[Section~3.1]{Sapozhnikov14}, where $\set_{L_0}$ is used instead of $\set$.
($\set_{L_0}$ is the set of vertices from $\set$ that belong to connected components of diameter at least $L_0$.)
In this note, we only consider connected $\set$, thus $\set_{L_0} = \set$.
\item
All the conditions on the scales $l_n$ and $r_n$ in \cite[Lemma~3.3 and Theorem~3.8]{Sapozhnikov14} are satisfied if $\sum_{j=0}^\infty\frac{r_j}{l_j}$ is sufficiently small.
This can be achieved by making $\lambda$ large enough.
\item
Assumptions (a) and (b) in the statement of Theorem~\ref{thm:isop:QKs} are identical to those in \cite[Definition~3.7]{Sapozhnikov14}.
An additional assumption (c) is imposed so that (b) and (c) imply that
$\cemax$ is connected and coincides with $\cemax_{K,s,L_0}$ from \cite[Definition~3.5]{Sapozhnikov14}.
\end{enumerate}
\end{remark}

Theorem~\ref{thm:isop:QKs} controls the size of the boundary of sets
larger than $L_s^{d(d+1)}$. For smaller sets, the following corollary will
be useful.

\begin{corollary}
  \label{cor:isop}
  In the setting of Theorem~\ref{thm:isop:QKs}, assume in addition that
  $K\geq L_s^{d^3}$. Then for every $\omega\in\Omega $ satisfying (a)--(c) of Theorem~\ref{thm:isop:QKs} and
  for all $A\subset {\cemax}$ with $|A|\leq \frac12\,{|\cemax|}$,
\[
|\partial_{\cemax} A| \geq \gamma|A|^{1-\frac 1d + \frac{1}{d^2}}((K+4)\,L_s)^{-\frac{1}{d}}.
\]
\end{corollary}

\begin{proof}
  For $A$ with
  $L_s^{d(d+1)} \le |A| \le \frac12\,|\cemax|\le
  ((K+4)L_s)^d$, the claim follows from Theorem~\ref{thm:isop:QKs}. On the other hand,
  for $A$ with $|A|\leq L_s^{d(d+1)}$,
  \[
    |\partial_\cemax A|\geq 1 =
    \frac{L_s^{d(d+1)(1-\frac1d + \frac{1}{d^2})}}
    {L_s^{(d^3+1)\frac{1}{d}}}
    \geq |A|^{1-\frac 1d + \frac{1}{d^2}} ((K+4)L_s)^{-\frac{1}{d}},
  \]
  as required.
\end{proof}

\subsection{Coupling with random interlacements}

Another principal ingredient for the proofs of our main results is the
coupling of $\range$ with the random interlacements inside of macroscopic
subsets of the torus which was constructed in \cite{CernyTeixeira14}. We use
here $\I^u$ to denote the random interlacements on $\mathbb Z^d$ at level
$u$ as introduced in \cite[(1.53)]{SznitmanAM}. In the next theorem, we
identify the torus $\torus$ with the set $[0,N)^d\cap \mathbb Z^d$.

\begin{theorem}\label{thm:coupling}
  \cite[Theorem~4.1]{CernyTeixeira14} Let $d\geq 3$, $u>0$. For any $\varepsilon>0$
  and $\alpha\in(0,1)$, there exist $\delta_{\ref{thm:coupling}}>0$,
  $C_{\ref{thm:coupling}}<\infty$, and a coupling $\mathbb Q$ of $\range$,
  $\I^{u(1-\varepsilon)}$, and $\I^{u(1+\varepsilon)}$, such that for all
  $N\geq 1$,
  \[
    \mathbb Q\left[\I^{u(1-\varepsilon)}\cap[0,\alpha N]^d
      \subseteq\range\cap[0,\alpha N]^d
      \subseteq \I^{u(1+\varepsilon)}\right]
    \geq 1 - C_{\ref{thm:coupling}}e^{-N^{\delta_{\ref{thm:coupling}}}}.
  \]
\end{theorem}

\section{Proof of the isoperimetric inequality} 

We may now proceed to the proof of Theorem~\ref{thm:isopineq:range}. To
this end, we need to check that assumptions (a)--(c) of
Theorem~\ref{thm:isop:QKs} hold true with high probability.
We fix $u>0$ and consider the function
\[
  \eta(u) = 1 - e^{-\frac{u}{g(0,0)}},
\]
where $g(\cdot,\cdot)$ is the Green function of the simple random walk on
$\Z^d$. (The function $\eta(u)$ is the density of random interlacements
  at level $u$.) We further fix $\varepsilon >0$ small enough so that
\begin{equation}
  \label{eq:etau}
  \eta_1 := \frac34\eta (u(1-\varepsilon ))
  \text{ and }
  \eta_2 := \frac54\eta (u(1+\varepsilon ))
  \text{ satisfy condition~\eqref{eq:etas}}.
\end{equation}
The first lemma provides an estimate on the probability that a vertex
is $s$-bad. Its proof relies on corresponding results for random
interlacements \cite[Lemmas~4.2 and 4.4]{DRS12} and the coupling from
Theorem~\ref{thm:coupling}. In its statement, we consider $\range$ as a
subset of $\Z^d$ obtained by the canonical periodic embedding of $\torus$
in $\Z^d$.

\begin{lemma}\label{l:Guxk}
  For any $u>0$, $\alpha\in(0,1)$, and $\varepsilon$ as in \eqref{eq:etau}, there exist
  $C_{\ref{l:Guxk}} = C_{\ref{l:Guxk}}(u,\varepsilon,\alpha)<\infty$ and
  $C_{\ref{l:Guxk}}' = C_{\ref{l:Guxk}}'(u,\varepsilon,\alpha,\lscale)<\infty$
  such that for all $\lscale\geq C_{\ref{l:Guxk}}$,
  $L_0\geq C_{\ref{l:Guxk}}'$, and $s\geq 0$ with $L_s+2L_0\leq \alpha N$,
  \[
    \mathbb P\left[\text{$0$ is $s$-bad in
        $\range$}\right] \leq 2\cdot 2^{-2^s}
    + C_{\ref{thm:coupling}} e^{-N^{\delta_{\ref{thm:coupling}}}}.
  \]
\end{lemma}

\begin{proof}[Proof of Lemma~\ref{l:Guxk}]
  Observe first that the event $\left\{\text{$0$ is $s$-bad in $\range$}\right\}$ depends only on the state of vertices inside of
  $B:=[-L_0,L_s+L_0]^d\cap \mathbb Z^d$. By assumption,
  $L_s+2L_0\le \alpha N$. Thus, using Theorem~\ref{thm:coupling}, we can
  couple $\range$ with $\I^{u(1\pm\varepsilon)}$ so that
  \[
    \P\left[\I^{u(1-\varepsilon)}\cap B\, \subseteq\, \range\cap B\,
      \subseteq\, \I^{u(1+\varepsilon)}\right]\geq 1 - C_{\ref{thm:coupling}} e^{-N^{\delta_{\ref{thm:coupling}}}}.
  \]
  Further, by the monotonicity of $(sa)$ and $(sb)$ bad events, the following inclusion holds:
  \begin{multline*}
    \left\{\text{$0$ is $s$-bad for the realization of $\range$},~
      \I^{u(1-\varepsilon)}\cap B\subseteq \range\cap B
      \subseteq \I^{u(1+\varepsilon)}\right\}\\
    \subseteq
    \left\{\text{$0$ is $(sa)$-bad for the realization of
        $\I^{u(1-\varepsilon)}$}\right\}\\
    \cup\left\{\text{$0$ is $(sb)$-bad for the realization of
        $\I^{u(1+\varepsilon)}$}\right\}.
  \end{multline*}
  By \cite[Lemmas~4.2 and 4.4]{DRS12}, the probabilities of the two events in the right hand side are bounded from above
  by $2\cdot 2^{-2^s}$.
\end{proof}

The next lemma implies that the assumptions of
Theorem~\ref{thm:isop:QKs} hold with a large probability for
$\set = \range$.

\begin{lemma}\label{l:proba}
  For each $u>0$, $\alpha\in(0,1)$, and $\varepsilon$ as in \eqref{eq:etau},
  there are
  $C_{\ref{l:proba}} = C_{\ref{l:proba}}(u,\varepsilon, \alpha)<\infty$ and
  $\delta_{\ref{l:proba}} = \delta_{\ref{l:proba}}(u,\varepsilon, \alpha)>0$
  such that for all $\lscale\geq C_{\ref{l:Guxk}}$,
  $L_0\geq C_{\ref{l:Guxk}}'$,  $s\geq 0$, and $K\ge L_s$ with
  $(K+4)L_s\leq \alpha N$
  \begin{equation}\label{eq:proba}
    \mathbb P\left[\begin{array}{c}\text{realization of $\range$ does not satisfy}
        \\ \text{any of (a)--(c) in Theorem~\ref{thm:isop:QKs}}\end{array}\right]
    \leq
    C_{\ref{l:proba}}\cdot (KL_s)^d\cdot \left(2^{-2^s}
      + e^{-L_s^{\delta_{\ref{l:proba}}}}\right).
  \end{equation}
\end{lemma}

\begin{proof}[Proof of Lemma~\ref{l:proba}]
  By Lemma~\ref{l:Guxk} and translation invariance,
  \[
    \mathbb P\left[\begin{array}{c}
        \text{realization of $\range$ does not satisfy}\\
        \text{(a) in Theorem~\ref{thm:isop:QKs}}\end{array}\right]
    \leq (K+4)^d\cdot \left(2\cdot 2^{-2^s} + C_{\ref{thm:coupling}}
      e^{-N^{\delta_{\ref{thm:coupling}}}}\right).
  \]

  By \cite[(1.65)]{SznitmanAM}, there exists $c>0$ such that for any $u>0$
   $\varepsilon\in(0,\frac 12)$, and  $m\ge 1$,
  \begin{equation}
    \label{eq:emptyinter}
    \mathbb P[\I^{u(1-\varepsilon)}\cap [0,m)^d = \emptyset] \leq
    e^{-cm^{d-2}u}.
  \end{equation}
  Choosing $m=L_s$ and combining this fact with
  Theorem~\ref{thm:coupling}, we obtain that
  \[
    \mathbb P\left[\begin{array}{c}
        \text{realization of $\range$ does not satisfy}\\
        \text{(c) in Theorem~\ref{thm:isop:QKs}}
    \end{array}\right]
    \leq (KL_s)^d\cdot e^{-cuL_s^{d-2}}  + C_{\ref{thm:coupling}}
    e^{-N^{\delta_{\ref{thm:coupling}}}}.
  \]

  To estimate the probability that (b) does not occur, it is not enough to
  use the coupling from Theorem~\ref{thm:coupling} and corresponding
  random interlacements results because the event in (b) is not monotone.
  We thus need to adapt the techniques of \cite[Section~8]{CernyPopov12}.

  We first
  claim that there are large $C=C(u)$ and small
  $\delta=\delta (u)\in (0,1)$
  such that
  \begin{equation*}
    \mathbb P\big[\text{ $\exists\, x,y\in [0,\delta L_s]^d \cap \range$
        s.t.~$x,y$ are not connected in $\range \cap \ballZ(x,L_s)$}\big]
    \le  C e^{- L_s^\delta }.
  \end{equation*}
  Indeed, this can be proved as \cite[Lemma~8.1]{CernyPopov12}, replacing
  the box of size $\ln^\gamma N$ used there with the box of size
  $\delta L_s$.  The proof in \cite{CernyPopov12} uses ingredients
  from \cite{TeixeiraWindisch11}, namely Lemmas~3.9, 3.10 and~4.3, which
  hold true for boxes up to size $N^{\frac12}$. Since
  $\delta L_s \le N^{1/2}$, by the assumption on $K$, we can use them
  without modifications.

  Using the translation invariance,
  \begin{equation*}
    \mathbb P\left[\parbox{8cm}{$\exists\,x,y\in \range \cap [0,KL_s)^d$ s.t.
        $|x-y|_\infty\le \delta L_s $ and
        $x,y$ are not connected in $\range \cap \ballZ(x,L_s)$}\right]
    \le C (KL_s)^d e^{-L_s^\delta }.
  \end{equation*}
  Moreover, using \eqref{eq:emptyinter} with $m=\delta L_s$,
  \begin{equation*}
    \mathbb P\big[\text{$\exists\,x\in \range \cap [0,KL_s)^d$ s.t.
        $\big(x+[0,\delta L_s)^d\big)\cap \range = \emptyset$}\big]
    \le C (KL_s)^d e^{-c u L_s^{d-2}}.
  \end{equation*}
  Finally, as in the proof of \cite[Theorem~1.6]{CernyPopov12}, assuming
  that the events of the last two displays hold, then
  for every
  $x,y\in \range \cap[0,KL_s)^d$ such that $|x-y|_\infty\le L_s$ we can
  find a sequence $x=x_0,\dots,x_k=y$ with $k\le 2\delta^{-1}$,
  $x_i\in \range \cap [0,KL_s)^d$, and
  $|x_i-x_{i-1}|_\infty\le \delta L_s$ for
  all $1\le i\le k$. In particular also $x_{i-1}$ and $x_i$ are connected
  in $\range \cap \ballZ(x_i,L_s)$ and
  thus $x,y$ are connected in $\range \cap \ballZ(x,2L_s)$. It follows that
  \[
    \mathbb P\left[\begin{array}{c}
        \text{realization of $\range$ does not satisfy}\\
        \text{(b) in Theorem~\ref{thm:isop:QKs}}\end{array}\right]
    \le C (KL_s)^d e^{-L_s^\delta }.
  \]
  By combining the three bounds and using the relation $L_s \leq \alpha N$,
  we obtain the desired bound \eqref{eq:proba}.
\end{proof}

\bigskip

We will prove Theorem~\ref{thm:isopineq:range} by making a suitable
choice of $s$ and $K$ in Lemma~\ref{l:proba} as functions of $N$.

\begin{proof}[Proof of Theorem~\ref{thm:isopineq:range}]
  Fix $N\geq 1$, $u>0$, $\mu\in(0,1)$, $\varepsilon>0$ satisfying
  \eqref{eq:etau}, and $\alpha = \frac{6}{7}$. It suffices to
  consider $\mu= \frac12$. Indeed, if $\mu>\frac12$, then the
  isoperimetric inequality for sets $A\subset\range$ with
  $\frac12|\range|\leq |A|\leq \mu|\range|$ follows from the
  isoperimetric inequality for $\range\setminus A$, see, e.g.,
  \cite[Remark~5.2]{Sapozhnikov14}.

Take the scales as in \eqref{def:scales} with
  $\lscale\geq C_{\ref{l:Guxk}}$ and $L_0\geq C_{\ref{l:Guxk}}'$. Without
  loss of generality, we assume that $N\geq 7L_0^{d^3+1}$. Let
  \[
    s = \max\left\{s'\geq 0~:~ L_{s'}^{d^3+1}
      \leq \frac{N}{7}\right\}
    \qquad\text{and}\qquad
    K = \min\left\{K'\geq 1~:~ K'L_s\geq \frac{N}{7}\right\}.
  \]
Notice that $K\geq L_s^{d^3}$ and $(K+4)L_s \leq \frac{6N}{7}$.
Thus, the parameters $s$ and $K$ satisfy the conditions of Corollary~\ref{cor:isop} and Lemma~\ref{l:proba}.
To apply Corollary~\ref{cor:isop}, for each $x\in [0,N)^d$, we define the local enlargement of $\range\cap(x+[0,KL_s)^d)$ by
\[
\cemax_x = \left\{y\in\range~:~\begin{array}{c}\text{$y$ is connected to some $z\in\range\cap(x+[0,KL_s)^d)$}\\ \text{by a path in $\range\cap \ballZ(z,2L_s)$}\end{array}\right\}.
\]
Here as before, we consider $\range$ as a subset of $\Z^d$ obtained by the canonical periodic embedding of $\torus$
in $\Z^d$. By Corollary~\ref{cor:isop}, Lemma~\ref{l:proba}, and translation invariance, for some $\beta =\beta(u)>0$ and $\gamma=\gamma(u)>0$,
\begin{equation}\label{eq:isop:allN7}
\mathbb P\left[\begin{array}{c}
\text{for all $x\in [0,N)^d$, $|\range\cap(x+[0,KL_s)^d|\geq \beta N^d$, and}\\
\text{for all $A\subset \cemax_x$ with $|A|\leq \frac12\,|\cemax_x|$,
$|\partial_{\cemax_x} A|\geq \gamma\cdot |A|^{1 - \frac{1}{d} + \frac{1}{d^2}}\cdot N^{-\frac{1}{d}}$}
    \end{array}\right]
    \geq 1- N^d\cdot \mathrm{RHS}_{\ref{l:proba}}\,,
\end{equation}
where
\[
\mathrm{RHS}_{\ref{eq:proba}} = C_{\ref{l:proba}}\cdot (KL_s)^d\cdot \left(2^{-2^s} + e^{-L_s^{\delta_{\ref{l:proba}}}}\right)
\]
is the right hand side of \eqref{eq:proba}.
The proof of Theorem~\ref{thm:isopineq:range} will be completed once we prove that for all $N\geq N_0(u)$, (a) the event in \eqref{eq:isop:allN7} implies
the event in \eqref{eq:isopineq:range}, with a possibly different $\gamma$,
and (b) $N^d\cdot \mathrm{RHS}_{\ref{eq:proba}}\leq e^{-(\log N)^2}$.

We begin showing (a). Assume that the event in \eqref{eq:isop:allN7} occurs. Let $A$ be a subset of $\range$ with $|A|\leq \frac 12|\range|$.
For each $x\in[0,N)^d$, let $A_x = A\cap \cemax_x$.

We choose points $x_1,\ldots,x_{7^d}\in [0,N)^d$ such that
\[
\range = \bigcup_{i=1}^{7^d} \cemax_{x_i}.
\]
Then $A = \cup_i A_{x_i}$ and $|A|\leq \sum_i|A_{x_i}|$.
Assume first that for all $i$, $|A_{x_i}|\leq \frac 12|\cemax_{x_i}|$.
Then,
\[
|\partial_{\cemax_{x_i}} A_{x_i}|\geq \gamma\cdot |A_{x_i}|^{1 - \frac{1}{d} + \frac{1}{d^2}}\cdot N^{-\frac{1}{d}}.
\]
Since for each $i$, $\partial_{\cemax_{x_i}} A_{x_i}\subset \partial_\range A$, we obtain that
\[
|\partial_\range A| \geq \frac{1}{7^d}\,\sum_i|\partial_{\cemax_{x_i}} A_{x_i}| \geq \frac{1}{7^d}\,\sum_i\gamma\cdot |A_{x_i}|^{1 - \frac{1}{d} + \frac{1}{d^2}}\cdot N^{-\frac{1}{d}}
\geq  \frac{1}{7^{d}}\,\gamma\cdot |A|^{1 - \frac{1}{d} + \frac{1}{d^2}}\cdot N^{-\frac{1}{d}},
\]
where in the last step we used the inequality $\sum_i |A_{x_i}|^{1 - \frac{1}{d} + \frac{1}{d^2}} \geq \left(\sum_i|A_{x_i}|\right)^{1 - \frac{1}{d} + \frac{1}{d^2}}$.
Thus, in this case, the event in \eqref{eq:isop:allN7} implies the event in \eqref{eq:isopineq:range} with $\gamma_{\ref{eq:isopineq:range}} = \frac{1}{7^{d}}\gamma$.

It remains to consider the case when for some $x$, $|A_x|>\frac12|\cemax_x|$.
We claim that in this case for $N\geq N_0(u)$, there exist $x$ such that
\begin{equation}\label{eq:x}
\frac 12 \beta N^d \leq |A_x|\leq \left(1 - \frac{\beta}{2\cdot 7^d}\right)|\cemax_x|.
\end{equation}
Assume that it is not the case. Since there exists $x$ such that
$|A_x|>\frac12|\cemax_x|$ and $|\cemax_x| \geq \beta N^d$, the
non-validity of \eqref{eq:x} implies that
$|A_x| > \left(1 - \frac{\beta}{2\cdot 7^d}\right) |\cemax_x|$. Assume
that the last inequality holds for all $x$. Then,
\[
|\range\setminus A| \leq \sum_i|\cemax_{x_i} \setminus A_{x_i}| < \frac{\beta}{2\cdot 7^d} \sum_i|\cemax_{x_i}|
\leq \frac{\beta}{2\cdot 7^d} \cdot 7^d|\range| \leq \frac12 |\range|,
\]
which contradicts the assumption that $|A|\leq \frac 12 |\range|$.
Thus, for each $x$, either $|A_x| < \frac12 \beta N^d$ or $|A_x| > \left(1 - \frac{\beta}{2\cdot 7^d}\right) |\cemax_x|$,
and both types exist. In particular, there exist $x,y$ with $|x-y|_1 = 1$ such that
$|A_x| < \frac12\beta N^d$ and $|A_y| > \left(1 - \frac{\beta}{2\cdot 7^d}\right) |\cemax_y|\geq \beta\left(1 - \frac{\beta}{2\cdot 7^d}\right)N^d$.
For these $x$ and $y$, on the one hand,
\[
|A_y\setminus A_x|\geq |A_y| - |A_x| \geq \frac13\beta N^d,
\]
and on the other,
\[
|A_y\setminus A_x|\leq |\cemax_y\setminus \cemax_x| \leq 2\left(\left((K+4)L_s\right)^d - \left(KL_s\right)^d \right)
\leq \frac{16d}{K}\,(KL_s)^d,
\]
where the last inequality holds for large enough $K$. Since $KL_s\leq N$ and $K\geq \sqrt{\frac N7}$,
the two bounds for $|A_y\setminus A_x|$ cannot be fulfilled simultaneously if $N\geq N_0(u)$.
This contradiction proves \eqref{eq:x}.

Let $x\in[0,N)^d$ satisfy \eqref{eq:x}.
Either $|A_x|\leq \frac12|\cemax_x|$ or $|\cemax_x\setminus A_x|\leq \frac12|\cemax_x|$. Thus,
\[
|\partial_{\cemax_x} A_x| \geq \gamma\cdot \min\left(|A_x|,|\cemax_x\setminus A_x|\right)^{1 - \frac{1}{d} + \frac{1}{d^2}}\cdot N^{-\frac{1}{d}}.
\]
Since $\partial_\range A\supseteq\partial_{\cemax_x} A_x$ and
\[
\min\left(|A_x|,|\cemax_x\setminus A_x|\right)\geq \min\left(\frac12\beta N^d,\frac{\beta}{2\cdot7^d}|\cemax_x|\right)\geq \frac{\beta^2}{2\cdot 7^d} N^d \geq \frac{\beta^2}{7^d}|A|,
\]
we obtain that $|\partial_\range A| \geq \frac{\beta^2}{7^d}\gamma\cdot |A|^{1 - \frac{1}{d} + \frac{1}{d^2}}\cdot N^{-\frac{1}{d}}$.
Thus, if \eqref{eq:x} holds, then the event in \eqref{eq:isop:allN7} implies
the event in \eqref{eq:isopineq:range} with $\gamma_{\ref{eq:isopineq:range}} = \frac{\beta^2}{7^d}\gamma$.
Putting the two cases together gives
  \[
    \mathbb P\left[\begin{array}{c}
        \text{for any $A\subset \range$ with
          $|A|\leq \frac12 |\range|$,}\\
        \text{$|\partial_\range A|\geq \frac{\beta^2}{7^d}
          \gamma\cdot |A|^{1 - \frac{1}{d} + \frac{1}{d^2}}\cdot N^{-\frac{1}{d}}$}
    \end{array}\right]
    \geq 1- N^d\cdot \mathrm{RHS}_{\ref{eq:proba}}.
  \]

\medskip

It remains to prove that $N^d\cdot \mathrm{RHS}_{\ref{eq:proba}}\leq e^{-(\log N)^2}$ for $N\geq N_0(u)$.
By \eqref{def:scales},
  \[
    \frac12\cdot N^{\frac{1}{d^3+1}}
    \leq \left(\frac{N}{7}\right)^{\frac{1}{d^3+1}}
    < L_{s+1}
    = l_s\cdot L_s \leq (4l_{s-1})^2\cdot L_s\leq 16\cdot L_s^3.
  \]
  Thus, $L_s\geq \frac{1}{4}\cdot N^{\frac{1}{3(d^3+1)}}$. On the other
  hand, by \eqref{def:scales},
  $L_s\leq L_0\cdot \lscale^{2s}\cdot 4^{s^3}$, which implies that there
  exists $c = c(L_0,\lscale)>0$ such that $s\geq c(\log N)^{\frac 13}-1$
  and $2^s \geq c(\log N)^4$. Thus, there exists $C = C(u)<\infty$ such that
$(KL_s)^d\cdot \left(2^{-2^s} + e^{-L_s^{\delta_{\ref{l:proba}}}}\right)\leq C e^{-(\log N)^3}$.
By taking $N$ large enough,
\begin{equation}\label{eq:rhs:proba}
N^d\cdot \mathrm{RHS}_{\ref{eq:proba}} < e^{-(\log N)^2}.
\end{equation}
The proof of Theorem~\ref{thm:isopineq:range} is complete.
\end{proof}

\section{Proof of Theorem~\ref{thm:mixingtime}} 

We begin with the proof of the upper bound.
It is very similar to the proof of
\cite[Theorem~3.1]{ProcacciaShellef}, which relies on the bound on the
mixing time from \cite[Theorem~1]{MorrisPeres}. For $r>0$, let
\[
  \phi(r) = \inf\left\{\frac{|\partial_\range A|}{|A|}~:~A\subset\range,
    0<|A|\leq \min\left\{r,\left(1 - \frac{1}{4d}\right)|\range|\right\}\right\}.
\]
By \cite[(16)]{ProcacciaShellef}, there exists $C = C(d)<\infty$ such that
\[
  \mix(\range) \leq C\int_{1}^{32dN^d}\frac{dr}{r\phi(r)^2}.
\]
Consider the event from \eqref{eq:isopineq:range} for $\mu = (1-\frac{1}{4d})$.
For each realization of $\range$ from this event and all $r$,
$\phi(r) \geq \gamma\cdot N^{-\frac{1}{d}}\cdot r^{-\frac{d-1}{d^2}}$.
Thus,
\[
  \int_{1}^{32dN^d}\frac{dr}{r\phi(r)^2} \leq \frac{1}{\gamma^2}
  \cdot N^{\frac{2}{d}}\cdot \int_{1}^{32dN^d}r^{\frac{2(d-1)}{d^2} - 1}dr
  \leq C\cdot N^2.
\]
By \eqref{eq:isopineq:range}, there exists $C = C(u)<\infty$ such that
\[
  \mathbb P\left[\mix(\range)\leq CN^2\right]\geq 1 - Ce^{-(\log N)^2}.
\]

\medskip
We proceed with the proof of the lower bound.
By the volume bound in \eqref{eq:isop:allN7} and \eqref{eq:rhs:proba}, there exist $C=C(u)<\infty$ and $\beta = \beta(u)>0$
such that for all $N\geq 1$,
\[
\P\left[\text{for all $x\in [0,N)^d$, $|\range\cap(x+[0,\frac N7)^d)|\geq \beta N^d$}\right]\geq 1 - C e^{-(\log N)^2}.
\]
Assume the occurrence of event under the probability.
By \cite[Theorem~2.1]{BarlowPerkins}, there exists $C=C(u)<\infty$ such that for each $\varepsilon>0$ and $n< \frac13 N$,
one can find $x=x(\varepsilon,n)\in \range$ so that
\[
\sum_{y\in \ballZ(x,n)} p_{\lfloor \varepsilon n^2\rfloor }(x,y) \geq 1 - C\,\varepsilon.
\]
On the other hand,
\[
\sum_{y\in \ballZ(x,n)}\pi(y) \leq \frac{2d (2n + 1)^d}{\beta N^d}.
\]
We take $\varepsilon$ small enough and $n<\varepsilon N$ so that the
first sum is larger than $\frac12$ and the second smaller than $\frac14$. Then there exists at
least one $y\in \ballZ(x,n)$ such that
$|p_{\lfloor \varepsilon n^2\rfloor}(x,y) - \pi(y)| \geq \frac14\pi(y)$.
Thus, $\mix(\range) \geq \varepsilon^3 N^2$, and we conclude that for some
$c=c(u)>0$ and $C=C(u)<\infty$,
\[
  \mathbb P\left[\mix(\range)\geq cN^2\right]\geq 1 - Ce^{-(\log N)^2}.
\]
The proof of Theorem~\ref{thm:mixingtime} is complete.
\qed

\medskip

\begin{remark}
  \begin{enumerate}[(a)]
    \item
    In the proof of the lower bound on $\mix(\range)$ we only used that
    $\range$ has positive density in large subboxes of $\torus$. This
    follows from the facts that $\range$ dominates random interlacements
    and the random interlacements are dense in large boxes. Both facts
    hold with probability $\geq 1 - Ce^{-N^\delta}$. Thus,
    $\mathbb P\left[\mix(\range)\geq cN^2\right]\geq 1 - Ce^{-N^\delta}$.

    \item
    The method of this note also applies (with minimal changes) to the
    largest connected component of the vacant set of the range
    $\mathcal V^u_N = \torus\setminus\range$, when $u$ is strongly
    supercritical, see \cite[Definition~2.4]{TeixeiraWindisch11}. For instance,
    the property (b) of Theorem~\ref{thm:isop:QKs} for the
    largest cluster is shown to be very likely for strongly supercritical
    $u$'s in \cite[Section~2.5]{DRS12}. So far, it is only known that
    strongly supercritical $u$'s exist if $d\geq 5$, see \cite{Teixeira11}.

    \item
    It is natural to consider $\range$ as a random subgraph of $\torus$ with
    edges traversed by the random walk. All our results remain true in
    this case. The proofs presented in the note are robust to this
    change, but the external ingredients should be adapted to
    corresponding bond models. Although the changes needed are only
    notational, presenting them would deviate us from the main goal of
    this note.
  \end{enumerate}
\end{remark}


\providecommand{\bysame}{\leavevmode\hbox to3em{\hrulefill}\thinspace}
\providecommand\MR{}
\renewcommand\MR[1]{\relax\ifhmode\unskip\spacefactor3000
\space\fi \MRhref{#1}{#1}}
\providecommand\MRhref{}
\renewcommand{\MRhref}[2]%
{\href{http://www.ams.org/mathscinet-getitem?mr=#1}{MR #2}}
\providecommand{\href}[2]{#2}


\end{document}